\documentclass[9pt]{amsart}
\usepackage{amssymb} 

\title{A LEFT TOPOLOGICAL MONOID  ASSOCIATED TO A TOPOLOGICAL GROUPOID}
\author{Habib Amiri}
\date{}

\newtheorem{thm}{Theorem}[section]
\newtheorem{corollary}[thm]{Corollary}
\newtheorem{lem}[thm]{Lemma}
\newtheorem{proposition}[thm]{Proposition}
\newtheorem{example}[thm]{Example}

\theoremstyle{definition}
\newtheorem{definition}[thm]{Definition}
\newtheorem{rem}[thm]{Remark}

\numberwithin{equation}{section}

\begin{document}
\thanks{}
\thanks{}
\maketitle
\vspace{-.5cm}
\centerline{Department of Mathematics, Faculty of Sciences,}
\centerline{Zanjan University, Zanjan, Iran}
\centerline{Email: h.amiri@znu.ac.ir}

\begin{abstract}
   This paper presents a fanctor $S$ from the category of groupoids to the category of semigroups. Indeed, a monoid  $S_G$ with  a right zero element is related to a topological groupoid $G$. The monoid $S_G$  is a subset of $C(G,G)$, the set of all continuous functions from $G$ to $G$, and with the compact- open topology inherited from C(G,G) is a left topological monoid. The group of units of $S_G$, which is denoted by $H(1)$, is isomorphic to a subgroup of the group of all bijection map from $G$ to $G$ under
composition of functions. Moreover, it is proved that $H(1)$ is embedded in the group of all invertible linear operators on $C(G)$, the set of all complex continuous function on $G$.

\vspace{.5cm}
\hspace{-.3cm}\textbf{2010 Mathematics Subject Classification:}{ 20D60, 18B40}

\hspace{-.3cm}\textbf{keywords:}{ Topological Groupoid; Semigroup}%

\end{abstract}


\section{Introduction}
If $G$ is a topological group, then each element $x$ of $G$ defines two translation operators $ L_x$ and $R_x$ on $C(G)$, the set of all complex valued  continuous functions on $G$,  by
$$L_xf(y)=f(x^{-1}y), R_xf(y)=f(yx).$$
The maps $x\mapsto L_x$ and $x\mapsto R_x$ are two monomorphisms, injective homomorphism, from $G$ into $\mathcal{L}\big(C(G)\big)$, the semigroup of all linear operators on $C(G)$ under composition of operators.
In the case where $G$ is a groupoid, since $x\in G$ is not composable with each element, it is not possible to define the operators $L_x$ and $R_x$ on $C(G)$ as in the group case. Of course,  $x\in G$ defines $L_x: C(G^{d(x)})\rightarrow C(G^{r(x)})$  by $L_xf(y)=f(x^{-1}y)$ for $f\in C(G^{d(x)})$. Similarly, $x\in G$ defines $R_x:C(G_{d(x)})\rightarrow C(G_{r(x)})$  by $R_xf(y)=f(yx)$ for $f\in C(G_{d(x)})$. It is easy to check that if $(x,y)$ is a composable pair, then $(L_xL_y)(f):=L_x(L_yf), (R_xR_y)(f):=R_x(R_yf)$ are well defined and $L_xL_y=L_{xy}, R_xR_y=R_{xy}.$

For a topological groupoid $G$, we introduce two monoids, semigroups with identity, $S_G$ and $S'_G$. Indeed, the map $G\mapsto S_G$ is a fanctor  from the category of groupoids to the category of semigroups. The elements of $S_G$ and $S'_G$ are chosen from $C(G,G)$, the set of continuous functions from $G$ to $G$.  The range and domain maps, $r$ and $d$, on the groupoid $G$  are  the identity element of $ S_G$ and $ S'_G$ respectively. The two monoids are isomorphic and have a common idempotent element $j$, the inverse map from $G$ to $G$, which is also a right zero for them.  We will show that $S_G$  with the compact- open topology inherited from C(G,G) is a left topological monoid.
The group of units, the maximal subgroup containing the identity element, of $S_G$ is obtained and is denoted by $H(1)$. The group $H(1)$ is isomorphic to a  subgroup of the group of all bijection map from $G$ to $G$ under composition of functions.
Also there exists an monomorphism  $f\mapsto L_f$ from $S_G$ into $\mathcal{L}\big(C(G)\big)$, where $\big(L_f(g)\big)(x)=g\big(f(x)x\big)$ for $g\in C(G)$ and $x\in G$. By this monomorphism the group $H(1)$ is embedded in  the  group of all invertible linear operators on $C(G)$. The monoid $S_G$ has a left cancellative submonoid $T_G$  which is embedded in the submonoid of all injective linear operators on $C(G)$.

\section{Definition and Notation}

The following definition is the definition of a groupoid given by P. Hahn in\cite{p.H}.

\begin{definition} A groupoid is a set $G$ endowed with a product map $(x,y)\mapsto xy :
G^2\rightarrow G$ where $G^2$ is a subset of $G\times G$ is called the set of composable pairs, and an
inverse map $x\mapsto x^{-1} : G\rightarrow G$ such that the following relations are satisfied:
\begin{enumerate}
\item For every $x\in G$, $(x^{-1})^{-1}=x$.
\item If $(x,y), (y,z)\in G^2$, then $(xy,z),(x,yz)\in G^2$ and $(xy)z=x(yz)$.
\item For all $x\in G$, $(x^{-1},x)\in G^2$ and if $(x,y)\in G^2$, then $x^{-1}(xy)=y$. Also for all $x\in G$,
$(x,x^{-1})\in G^2$ and if $(z,x)\in G^2$, then $(zx)x^{-1}=z$.
\end{enumerate}
\end{definition}

The maps $ r$  and $d$ on $G$
 defined by the formulae $r(x)=xx^{-1}$ and $d(x)=x^{-1}x$
are called the range and domain  maps.
It follows easily from the definition that they
have a common image called the {\it{unit space}} of $G$
 which is denoted by $G^0$. The pair
$(x,y)$ is composable if and only if $r(y)=d(x)$.

It is perhaps helpful to picture a groupoid as a collection of points with
various arrows connecting the points. For example, we can  write $(x\longrightarrow{^{\hspace{-.4cm}g}}\ \ \  y)$
to indicate that $g$ is an arrow with source $x$ and target $y$. By this notation, if $(g,h)\in G^2$, then $(x\longrightarrow{^{\hspace{-.4cm}g}}\ \ y\longrightarrow{^{\hspace{-.4cm}h}}\ \ z)\mapsto (x\longrightarrow{^{\hspace{-.6cm}gh}}\ \ z)$ by the product map.

Condition
(iii) implies that $r(x)x=x, xd(x)=x$. For $u, v\in G^0$, $G^u= r^{-1}(u), G_v= d^{-l}(v)$ and $G^u_v=G^u\cap Gv$. A groupoid $G$ is called a principal  groupoid if and only if the map $(r,d)$ from $G$ into $G^0\times G^0$ denoted by $(r,d)(x)=(r(x),d(x))$ is
one-to-one.

If $G, H$ are two groupoids.  A map $\phi :G\rightarrow H$
is an homomorphism if and only if  $(x,y)\in G^2$ implies $(\phi(x),\phi(y))$ is in $H^2$,
 and in this case, $\phi(xy)=\phi(x)\phi(y)$. Also  a map $\psi :G\rightarrow H$
is an antihomomorphism if and only if  $(x,y)\in G^2$ implies  $(\psi(y),\psi(x))$ is in $H^2$,
 and in this case, $\psi(xy)=\psi(y)\psi(x)$. It is easy to see that groupoid homomorphisms  and groupoid antihomomorphisms between two groupoids  map units to units and inverses to inverses.

The reader interested in
groupoids  is referred to the books \cite{pat} and \cite{re}.

 A topological groupoid consists of a groupoid $G$ and a topology
compatible with the groupoid structure. That is the inverse map $x\mapsto x^{-1}: G\rightarrow G$ is continuous, as well as the product map $(x,y)\mapsto xy: G^2\rightarrow G$ is continuous where $G^2$ has the induced topology from
$G\times G$. We are concerned with topological groupoids whose topology is Hausdorff and
locally compact. We call them {\it{locally compact Hausdorff groupoids}.}

If $S$ is a semigroup, then  a nonempty subset $T$ of $ S$ is called a left ideal of $ S$ if $ S.T\subset T$ and it is called a right ideal of $ S$ if $T. S\subset T$. If $T$ is both a left ideal and a right ideal of $ S$, then $T$ is called an ideal of $ S$. A left ideal(right ideal, ideal) of a semigroup $ S$ is said to be minimal if it properly contains no left ideal (right ideal, ideal, respectively) of $ S$. An element $e$ of $ S$ is said to be an idempotent if $e^2=e.e=e$. An element $z$ in a Semigroup $ S$ is a right zero if $s.z=z$ for all $s\in S$.
By Theorem 2.8 of \cite{Berg} for an idempotent $e$ in $ S$, the left ideal $ Se$ is minimal left ideal if and only if  $e S$ is a minimal right ideal and equivalently $e Se$ is a group. An injective homomorphism between two semigroups is called a{\it{monomorphism}} and  a semigroup with identity element is called a {\it{monoid}}.

\section{A monoid related to a topological groupoid}

For a topological groupoid $G$, we denote the set of all continuous function from $G$ into itself by $C(G,G)$.
\begin{definition}For a topological groupoid $G$, set

$$ S_G=\{f\in C(G,G)\  f(x)\in G_{r(x)}\ \mbox{for all}\  x\in G\},$$
$$ S'_G=\{f\in C(G,G)\  f(x)\in G^{d(x)}\ \mbox{for all}\  x\in G\}.$$
\end{definition}

Note that if $G$ is a topological group, then $ S_G= S'_G=C(G,G)$.
\begin{proposition} Two sets  $ S_G$ and $ S'_G$ with the following binary operations are two isomorphic monoid.
$$(f\ast g)(x)=g\big(f(x)x\big)f(x)\ \ \ \ f, g\in S, x\in G,$$
$$(h\star k)(x)=h(x)k\big(xh(x)\big)\ \ \ \ h, k\in S', x\in G.$$
The range and domain maps, $r$ and $d$, on the groupoid $G$  are the identity element of $ S_G$ and $ S'_G$ respectively.

\end{proposition}

\begin{proof}
 Note that if $f, g\in  S_G,\ h, k\in  S'_G$ and $x\in G$, then $\big(f(x),x\big), \big(x,h(x)\Big)\in G^2$ and
 $$d\Big(g\big(f(x)x\big)\Big)=r\big(f(x)x\big)=r(f(x)),$$
 $$d\big(h(x)\big)=d\big(xh(x)\big)=r\Big(k\big(xh(x)\big)\Big).$$ Hence $\Big(g\big(f(x)x\big),f(x)\Big), \Big(h(x),k\big(x h(x)\big)\Big)\in G^2$, consequently $f\ast g$ and $h\star k$ are well defined function from $G$ to $G$ and by the continuity assumption of $f$ and $g$, are continuous.  Also,
 $$d\Big(\big(f\ast g\big)(x)\Big)=d\Big(g\big(f(x)x\big)f(x)\Big)=d(f(x))=r(x),$$
 $$r\Big(\big(h\star k\big)(x)\Big)=r\Big(h(x)k\big(xh(x)\big)\Big)=r(h(x))=d(x).$$

 Therefore $f\ast g\in S_G$ and $h\star k\in S'_G$. So  $ S_G$ and $ S'_G$ are closed under these binary operations.

 Next we show that the binary operation of $ S_G$ is associative. The associativity of the binary operation of $ S'_G$ is similar. For $f, g, h\in S_G$ and $x\in G$,
$$\Big(f\ast \big(g\ast h\big)\Big)(x)=\Big(g\ast h\Big)\big(f(x)x\big)f(x)=\Big[h\Big(g\big(f(x)x\big)f(x)x\Big)g\big(f(x)x\big)\Big]f(x).$$
On the other hand
$$\Big(\big(f\ast g\big)\ast h\Big)(x)=\Big[h\Big(\big(f\ast g\big)(x)x\Big)\Big]\big(f\ast g\big)(x)=\Big[h\Big(g\big(f(x)x\big)f(x)x\Big)\Big]g\big(f(x)x\big)f(x)$$
So $f\ast (g\ast h)=(f\ast g)\ast h$.

The range map $r(x)=xx^{-1}$  belongs to  $ S_G$ and for $g\in  S_G$,
$$\big(r\ast g\big)(x)=g\big(r(x)x\big)r(x)=g(x)r(x)=g(x)d(g(x))=g(x)$$
$$\big(g\ast r\big)(x)=r\big(g(x)x\big)g(x)=r\big(g(x)\big)g(x)=g(x),$$
so $r\ast g=g\ast r=g$, and therefore $r$ is the identity of $ S_G$. Similarly, the domain map $d(x)=x^{-1}x$ is the identity of $ S'_G$. Therefore $S_G$ and $S'_G$ are two monoids.

 If for $f\in S_G$  we define $f^*(x)=\big(f(x^{-1})\big)^{-1}$, then $f^*$ is continuous and for each $x\in G$,
 $$r\big(f^*(x)\big)=r\Big(\big(f(x^{-1})\big)^{-1}\Big)=d(f(x^{-1}))=r(x^{-1})=d(x).$$
 Therefore $f^*\in S'_G$. Note that $(f^*)^*=f$ for every $f\in S_G$. We show that the map $f\mapsto f^*$ from $ S_G$ to $ S'_G$ is a semigroup isomorphism.
It is easy to check that this map is bijective. The proof will be completed if we prove that $(f\ast g)^*=f^*\star g^*$ for every $f, g\in S_G$.

 If $f, g\in S_G$, then for each $x\in G$,
$$\begin{array}{ll}
\big(f\ast g\big)^*(x)&=\Big(\big(f\ast g\big)(x^{-1})\Big)^{-1}\\
&=\Big(g\big(f(x^{-1})x^{-1}\big)f(x^{-1})\Big)^{-1}\\
&=\big(f(x^{-1})\big)^{-1}\Big(g\big(f(x^{-1})x^{-1}\big)\Big)^{-1}\\
&=f^*(x)\Big[g\Big(\big(xf^*(x)\big)^{-1}\Big)\Big]^{-1}\\
&=f^*(x)g^*\big(xf^*(x)\big)\\
&=\big(f^*\star g^*\big)(x).
\end{array}$$
\end{proof}

\begin{proposition}
The following assertions hold for the two monoids.
\begin{enumerate}
\item The inverse map $j(x)=x^{-1}$ from $G$ to $G$ is belong to $ S_G\cap S'_G$ and is an idempotent of these two semigroups which is also a right zero element for them.
\item The element $j$ is left zero for $ S_G$  if and only if $f(u)=u$ for every $f\in S_G$ and $u\in G^0$. Similar assertion holds for the semigroup $ S'_G$.
\item $ S_G\cap S'_G$ is a left ideal of $ S_G$ and $ S'_G$.

\item $j S_G$ is a minimal ideal of $S_G$, and $j S'_G$ is a minimal ideal for $ S'_G$.

\item If $\phi\in S_G$ is a bijection, then $\phi^{-1}$ is an element of $S_G'$, where $\phi^{-1}(\phi(x))=x=\phi(\phi^{-1}(x))$ for every $x\in G$.
\item The only injective antihomomorphism element of $S_G$ which is also an idempotent is the element $j$.

\item The only antihomomorphism element of $S_G$ which is also a right zero element is the element $j$.
\item If $f\in S_G\cap S_G'$ is an antihomomorphism, then $f\star f=f\ast f$.

\end{enumerate}
\end{proposition}
\begin{proof}
\begin{enumerate}
\item  For every $x\in G$,   $j(x)=x^{-1}\in G^{d(x)}_{r(x)}$. So $j$ is an element of $ S_G\cap S'_G$. Also
$$\big(j\ast j\big)(x)=j\big(j(x)x\big)j(x)=x^{-1}\big(j(x)\big)^{-1}j(x)=x^{-1}=j(x),$$
that is $j\ast j=j$. Since $j^*=j$, therefore $j\star j=(j\ast j)^*=j^*=j$.
For $f\in S_G,\ g\in S'_G$ and $x\in G$,
$$(f\ast j)(x)=j\big(f(x)x\big)f(x)=x^{-1}\big(f(x)\big)^{-1}f(x)=x^{-1}=j(x).$$
Consequently $g\star j=(g^*\ast j)^*=j$. Hence $j$ is a right zero for $ S_G$ and $ S'_G$.

\item Suppose that $j$ is a left zero for $ S_G$, for $f\in  S_G$ and  $x\in G$
$$x^{-1}=j(x)=(j\ast f)(x)=f\big(j(x)x\big)j(x)=f\big(d(x)\big)x^{-1}.$$
Therefore $f\big(d(x)\big)=d(x)$ for every $x\in G$. Conversely, if $f(u)=u$ for every $u\in G^0$, then for every $x\in G$
$$(j\ast f)(x)=f\big(j(x)x\big)j(x)=f\big(d(x)\big)x^{-1}= d(x)x^{-1}=x^{-1}=j(x).$$

\item  Let $g\in S_G\cap S'_G$ and $f\in S_G$, then
$$r\Big(\big(f\ast g\big)(x)\Big)=r\Big(g\big(f(x)x\big)f(x)\Big)=r\Big(g\big(f(x)x\big)\Big)=d\big(f(x)x\big)=d(x)$$
and
$$d\Big(\big(f\ast g\big)(x)\Big)=d\Big(g\big(f(x)x\big)f(x)\Big)=d\big(f(x)\big)=r(x).$$
So $f\ast g\in S_G\cap S'_G$.
Similarly, we can show that $ S_G\cap S'_G$ is a left ideal of $ S'_G$.

\item It is obvious that $j S_G$ is a right ideal of $ S_G$ and since $j$ is a right zero for $ S_G$,  $j S_G$ is also a left ideal of $ S_G$. Now let $f\in S_G$. If $ S_G(j\ast f) S_G=j S_G$, then by  Proposition 2.4 of \cite{Berg}, $j S_G$ is a minimal ideal of $ S_G$. However, since $j$ is a right zero for $ S_G$, $ S_G(j\ast f) S_G=(j\ast f) S_G\subset j S_G$. If $g\in S_G$, then $j\ast g=j\ast(j\ast f)\ast j\ast g\in  S_G(j\ast f) S_G$. So $j S_G\subset S_G(j\ast f) S_G$.
\item Let $\psi$ be the inverse of $\phi,\ y\in G$ and $\psi(y)=x$. Then $r\big(\psi(y)\big)=r\Big(\psi\big(\phi(x)\big)\Big)=r(x)=d\big(\phi(x)\big)=d(y)$, that is $\psi\in S_G'$.
\item Let $\phi$ be an idempotent which is also an injection antihomomorphism element of $S_G$, so $\phi\ast\phi=\phi$, that is, $\phi\big(\phi(x)x\big)\phi(x)=\phi(x)$ for all $x\in G$, so $\phi\big(\phi(x)x\big)=r\big(\phi(x)\big)=\phi\big(d(x)\big)$. The injectivity of $\phi$ implies that $\phi(x)x=d(x)$, so $\phi(x)=x^{-1}=j(x)$
\item By definition of a right zero element in a semigroup, $\phi\ast\psi=\psi$ for all $\phi\in S_G$. Therefore in a special case $j\ast\psi=\psi$ and consequently $\psi(j(x)x)j(x)=\psi(x)$. Hence $j(x)=\big[\psi(j(x)x)\big]^{-1}\psi(x)=\big[\psi\big(d(x)\big)\big]^{-1}\psi(x)=\big[r\big(\psi(x)\big)\big]^{-1}\psi(x)=r\big(\psi(x)\big)\psi(x)=\psi(x)$
   \item By the definition of the binary operations of $S_G, S'_G$, it is straightforward.

\end{enumerate}
\end{proof}

\begin{rem}
We consider  $C(G,G)$ with the compact-open topology. Recalling that the compact-open topology is the topology generated by the base consisting of all sets $\cap_{i=1}^k M(C_i,U_i)$, where  $C_i$ is a compact subset of $G$ and $U_i$ is an open
subsets of $G$ for $i=1, 2,\ldots ,k$ and where, $M(A,B)=\{f\in C(G,G): f(A)\subset B\}$ for $A, B\subset G$. The reader is referred to \cite{EN} for more details about this topology.
\end{rem}

\rm{In the following we will show that $S_G$ with the compact-open topology inherited from $C(G,G)$ is a left topological monoid. It is easy to check that the isomorphism $f\mapsto f^*$  is continuous from $S_G$ into $S'_G$, with compact-open topology, so $S'_G$ is left topological.}

\begin{proposition}
The monoid  $S_G$ with the compact-open topology is a left topological semigroup.
\end{proposition}
\begin{proof}
Let $f\in S_G$ and let $\{g_{\alpha}\}_{\alpha\in \Sigma}$ be a net in $S_G$ converging to $g\in S_G$ in compact-open topology. We will show that $f\ast g_{\alpha}\rightarrow f\ast g$ in compact-open topology.
Suppose that $\cap_{i=1}^{k}M(C_i,U_i)$ is a neighborhood of $f\ast g$ in compact-open topology. So $g\big(f(x)x\big)f(x)\in U_i$ for every $x\in C_i$ and $i=1, 2, \ldots, k$. Since $G$ is a topological groupoid, for $x\in C_i$ there exists tow open sets $U'_x$ and $V'_x$ with $g\big(f(x)x\big)\in U'_x$, $f(x)\in V'_x$ and $U'_xV'_x\subset U_i$. Let $U_x, V_x$ be two open set in $G$ with $g\big(f(x)x\big)\in U_x, f(x)\in V_x$ and $\overline{U_x}\subset U_x'$  and $\overline{V_x}\subset V_x'$.
The set $\Big\{\Big(g\big(f(x)x\big),f(x)\Big): x\in C_i\Big\}$ is a compact subset of $G^2$, hence there exist $x_1, x_2, \ldots, x_{n_i}$ in $C_i$ such that
$$\Big\{\Big(g\big(f(x)x\big),f(x)\Big): x\in C_i\Big\}\subset \bigcup_{j=1}^{n_i} \Big(U_{x_j}\times V_{x_j}\Big)\cap G^2.$$ So
$$\Big\{g\big(f(x)x\big)f(x): x\in C_i\Big\}\subset \bigcup_{j=1}^{n_i}U_{x_j}V_{x_j}\subset U_i.$$

Put $F_j=\{f(x)x: x\in C_i\ \mbox{and}\ g\big(f(x)x\big)\in \overline{U_{x_j}}\}$, then $F_j$ is a compact set for $j=1, 2, \ldots, n_i$ and $g(F_j)\subset U'_{x_j}$. So $\cap_{j=1}^{n_i}M(F_j, U_{x_j}')$ is a neighborhood of $g$ in compact-open topology. Therefore there exists $\beta\in\Sigma$ with $g_{\alpha}(F_j)\subset U_{x_j}'$ for every $\alpha\geq\beta$ and $j=1, 2, \ldots n_i$ and $i=1, 2, \ldots, k$. Now if $x\in C_i$, then there exists $j\in\{1, 2, \ldots n_i\}$ with $f(x)x\in F_j$, then $f(x)\in V'_{x_j}$. Therefore $\big(f\ast g_{\alpha}\big)(x)=g_{\alpha}\big(f(x)x\big)f(x)\in g_{\alpha}(F_j)f(x)\subset U'_{x_j}V'_{x_j}\subset U_i$ for every  $\alpha\geq\beta$, that is $\big(f\ast g_{\alpha}\big)(C_i)\subset \big(f\ast g\big)(U_i)$ for $\alpha\geq\beta$, and $i=1, 2, \ldots, k$. Therefore $f\ast g_{\alpha}\rightarrow f\ast g$ in compact-open topology.
\end{proof}

\begin{proposition}
The map $G\mapsto S_G$ from the category of groupoids to the category of semigroups is a fanctor. Therefore if
 $G$ and $H$ are two isomorphic groupoids, then $S_G$ and $S_H$ are two isomorphic monoids.
\end{proposition}
\begin{proof}
Suppose that $\psi: G\rightarrow H$ is a groupoid homomorphism. Define $\Psi: S_G\rightarrow S_H$ by $\big(\Psi(f)\big)(x)=(\psi f\psi^{-1})(x)$. We have
$$\begin{array}{ll}
d\Big(\big(\Psi(f)\big)(x)\Big)&=d\Big[\psi\Big(f\big(\psi^{-1}(x)\big)\Big)\Big]\\
&=\psi \Big[d\Big(f\big(\psi^{-1}(x)\big)\Big)\Big]\\
&=\psi \Big[r\big(\psi^{-1}(x)\big)\Big]\\
&=\psi \Big[\psi^{-1}\big(r(x)\big)\Big]\\
&=r(x).
\end{array}$$
Therefore $\Psi(f)\in S_H$. We will show that the map $\Psi$ is a semigroup homomorphism. Let $f, g\in S_G$ and $x\in H$,
 $$\begin{array}{ll}
 \big(\Psi({f\ast g})\big)(x)&=\psi\Big[\Big(f\ast g\Big)\big(\psi^{-1}(x)\big)\Big]\\
&=\psi\Big[g\Big(f\big(\psi^{-1}(x)\big)\psi^{-1}(x)\Big)f\big(\psi^{-1}(x)\big)\Big]\\
&=\psi\Big[g\Big(f\big(\psi^{-1}(x)\big)\psi^{-1}(x)\Big)\Big]\psi\Big[f\big(\psi^{-1}(x)\big)\Big]\\
&=\psi\Big[g\Big(\psi^{-1}\big((\psi f\psi^{-1})(x)\big)\psi^{-1}(x)\Big)\Big]\psi\Big[f\big(\psi^{-1}(x)\big)\Big]\\
&=(\psi g\psi^{-1})\Big[(\psi f\psi^{-1})(x)x\Big](\psi f\psi^{-1})(x)\\
&=\Psi(g)\Big[\Big(\Psi(f)\big)(x)\Big)x\Big]\big(\Psi(f)\big)(x)\\
&=\big(\Psi(f)\ast\Psi(g)\big)(x).
\end{array}$$
So the map $f\mapsto \Psi(f)$ is a semigroup homomorphism.
It is obvious that if $id_G:G\rightarrow G$ is the identity, then $\Psi({id_G}): S_G\rightarrow S_G$ is the identity. So the proof is completed.
\end{proof}

\begin{lem}
The monoid $S_G$ is isomorphic to a submonoid of the semigroup $C(G,G)$ under the binary operation $\big(f\circ g\big)(x)=g\big(f(x)\big)$ for $f,g\in C(G,G)$ and $x\in G$.
\end{lem}
\begin{proof}
For $\phi\in S_G$ define, $L_{\phi}: G\rightarrow G$ by $L_{\phi}(x)=\phi(x)x$. Then $$L_{\phi\ast\psi}(x)=\big(\phi\ast\psi(x)\big)x=\psi\big(\phi(x)x\big)\phi(x)x=L_{\psi}\big(\phi(x)x\big)=L_{\psi}\big(L_{\phi}(x)\big).$$
That is, the map $\phi\mapsto L_{\phi}$ is an homomorphism from $S_G$ into $C(G,G)$. It is obvious that this homomorphism  is injective. So $S_G$ is isomorphic to the subsemigroup $\{L_{\phi}: \phi\in S_G \}$ of the semigroup $\big(C(G,G),\circ\big)$.
\end{proof}

\begin{proposition}
For $\phi\in S_G$ denote the map $x\mapsto \phi(x)x$ by $L_{\phi}$, then the  set $$H(1)=\{\phi\in S_G: \mbox{the map}\ \ L_{\phi} \ \ \mbox{is a bijection} \}$$ is the group of units of $S_G$, the maximal subgroup of $S_G$ which containing the identity element $r$.
\end{proposition}
\begin{proof}
Obviously the map $r$ which is the identity of $S_G$ is belong to $H(1)$. Also by the previous Lemma the map $\phi\mapsto L_{\phi}$ from  the monoid $(S_G,\ast )$ to the monoid  $\big(C(G,G),\circ\big)$ is a homomorphism. Since the composition of two bijection map is bijective, $H(1)$ is a submonoid of $S_G$.  Now let $\phi\in H(1)$ and define $\psi(y)=\big(\phi(x)\big)^{-1}$, where $y=L_{\phi}(x)=\phi(x)x$, $\psi$ is well-defined. We will show that $\psi$ is the inverse of $\phi$, that is $\psi\ast\phi=\phi\ast\psi=r$. By definition of $\psi$, $\big(\phi\ast\psi\big)(x)=\psi\big(\phi(x)x\big)\phi(x)=\big(\phi(x)\big)^{-1}\phi(x)=d\big(\phi(x)\big)=r(x)$. To prove that $\psi\ast\phi=r$, note that $y=\phi(x)x$ implies that $\psi(y)y=\psi(y)\phi(x)x=\psi\big(\phi(x)x\big)\phi(x)x=\big(\phi\ast\psi (x)\big)x=r(x)x=x$. Therefore $\big(\psi\ast\phi\big)(y)=\phi\big(\psi(y)y\big)\psi(y)=\big(\psi(y)\big)^{-1}\psi(y)=d\big(\psi(y)\big)=r(y)$ for every $y\in G$. To complete the proof we just need to show that $H(1)$ is maximal. Let $K$ be a subgroup of $S_G$ which containing $r$.
 To prove that $K\subset H(1)$, it is enough to show that if $\phi\in K$, then the map $L_{\phi}$ is a bijective. Suppose that $\psi$ is the inverse of $\phi$, then $\phi\ast\psi=\psi\ast\phi=r$. Therefore $L_{\psi}\circ L_{\phi}=L_{\phi}\circ L_{\psi}=h_r=I$ the identity map on $G$. So $L_{\phi}$ is a bijection.
\end{proof}

 For $\phi\in S_G$, the set of all fixed point of $\phi$ is denoted by $Fix(\phi)$.
\begin{proposition} Let $A$ be a subset  of $G$ and put $I_A=\{\phi\in C(G,G) :\phi(A)\subset A \}$, then the following assertion hold,

 If $A$ is a subgroupoid  of $G$, then $S_A=I_A\cap S_G$ is a subsemigroup of $S_G$.

\end{proposition}
\begin{proof}
 It is easy to see that for a subgroupoid  $A$ of $G$, $\phi\in I_A$ if and only if $L_{\phi}\in I_A$. Now let $\phi, \psi\in S_A$, then by Lemma 3.7, $L_{\phi\ast\psi}(A)=L_{\psi}\big(L_{\phi}(A)\big)\subset L_{\psi}(A)\subset A$. So $S_A$ is a subsemigroup of $S_G$.
\end{proof}

\begin{rem}
 If $\phi: G\rightarrow H$ is a groupoid homomorphism, then $d\big(\phi(x)\big)=\phi\big(d(x)\big)$ and $r\big(\phi(x)\big)=\phi\big(r(x)\big)$ for all $x\in G$.  Note that the conditions $\phi\circ r=r\circ\phi$ and $\phi\circ d=d\circ\phi$  does not imply that $\phi$ is a homomorphism. For example in the case where $G$ and $H$ are two groups,  every function  $\phi: G\rightarrow H$ which preserves the identity element, satisfies in the two conditions. Similarly $d\big(\psi(x)\big)=\psi\big(r(x)\big)$ and $r\big(\psi(x)\big)=\psi\big(d(x)\big)$ for all $x\in G$ does not imply that $\psi$ is an antihomomorphism.
Now let $\phi\in C(G,G)$ and $Fix(\phi)$ be the fixed-point set of $\phi$. If $\psi\in S_G$ with $d\big(\psi(x)\big)=\psi\big(r(x)\big)$ for all $x\in G$, then  $\psi\big(r(x)\big)=d\big(\psi(x)\big)=r(x)$ and therefore $G^0\subset Fix(\psi)$. Conversely, if  $\psi\circ r=d\circ\psi$ and $G^0\subset Fix(\psi)$, then $\psi\in S_G$. Therefore for an element $\psi$ of $C(G,G)$ with $\psi\circ r=d\circ\psi$ , we have $\psi\in S_G$ if and only if $G^0\subset Fix(\psi)$

In a special case  if $\phi$ is a continuous antihomomorphism, then $\phi\in S_G$ if and only if $G^0\subset Fix(\phi)$.  In the following we obtain this result when the condition $G^0\subset Fix(\phi)$ is replaced by $\phi(G^u)\cap G_u\neq\emptyset$ for every $u\in G^0$ , where $\phi(G^u)$ is the image of $G^u$ under the map $\phi$.
\end{rem}
\begin{proposition}
Let $\phi\in C(G,G)$ and  $d\circ \phi=\phi\circ r$, then $\phi\in S_G$ if and only if  $\phi(G^u)\cap G_u\neq\emptyset$ for all $u\in G^0$.
\end{proposition}
\begin{proof}
Suppose that $\phi(G^u)\cap G_u\neq\emptyset$ for all $u\in G^0$. Therefore for $z\in G$ there exists $x\in G^{r(z)}$ with $\phi(x)\in G_{r(z)}$. So
$$\begin{array}{ll}
d(\phi(z))&=\phi(r(z))\\
&=\phi(r(x))\\
&=d(\phi(x))\\
&=r(z).
\end{array}$$ That is, $\phi\in S_G$. The converse is hold, since $d(\phi(x))=r(x)$ and $x\in G^{r(x)}$, that is $\phi(G^{r(x)})\cap G_{d(x)}\neq\emptyset$ for every $x\in G$.
\end{proof}
For $\phi\in S_G$ and $\psi\in S'_G$,  define $L_{\phi}(x)=\phi(x)x$ and ${\mathcal{R}}_{\psi}(x)=x\psi(x)$
\begin{proposition}
Let $ S_G$ and $ S'_G$ be the monoids which are defined in Definition $3.1$. Set
$$ T_G=\{\phi\in S_G: \{L_{\phi}(x) : x\in G\}\ \mbox{is dense in G}\},$$
$$ T'_G=\{\psi\in S'_G: \{\mathcal{R}_{\psi}(x): x\in G\}\ \mbox{is dense in G}\}.$$
Then $ T_G$  is a left cancellative   submonoid  of $ S_G$, $T'_G$ is a left cancellative submonoid  of $ S'_G$ and $ T_G$ is isomorphic to $T'_G$.
\end{proposition}
\begin{proof} It is obvious that $r\in T_G$ and $d\in T'_G$. Suppose that $\phi, \psi\in T_G$. Since the map
$L_{\psi}$ is continuous by using the density of the sets $\{L_{\phi}(x): x\in G\}$ and $\{L_{\psi}(x): x\in G\}$ in $G$, we obtain that the set $\{L_{\psi}\big(L_{\phi}(x)\big): x\in G\}$, which is equal to $\{L_{\phi\ast\psi}(x): x\in G\}$, is dense in $G$.
So $ T_G$ is a subsemigroup of $ S_G$. If $f, g, h\in S_G$ and $f\ast g=f\ast h$. Therefore $L_{g}\big(L_{f}(x)\big)=L_{h}\big(L_{f}(x)\big)$ for every $x\in G$. The density of $\{L_{f}(x): x\in G\}$ in $G$  and the continuity of $L_{g}$ and $L_{h}$ imply that $L_{h}=L_{g}$ and therefore $g=h$.

Similarly, we can show that $T'_G$ is a left cancellative subsemigroup of $S'_G$. To prove that $ T_G$ and $T'_G$ are isomorphic, using the proof of Proposition $3.2$,  it is enough to show that if $f\in T_G$, then $f^*\in T'_G$. Let $f\in T_G$, then
$$\begin{array}{ll}
\{f(x)x: x\in G\}^{-1}&=\{x^{-1}\big(f(x)\big)^{-1}: x\in G\}\\
&=\{t\big(f(t^{-1})\big)^{-1}: t\in G\}\\
&=\{tf^*(t): t\in G\}.
\end{array}
 $$
  The continuity of the inverse map from $G$ to $G$  implies  that $A\subset G$ is dense in $G$ if and only if $A^{-1}$ is dense in $G$. Therefore $f^*\in T'_G$.
\end{proof}

\begin{example}$($Transformation group groupoids $[4, p. 6] )$. Suppose that the group $T$ acts
on the space $U$ on the right. The image of the point $u\in U$ by the transformation $t\in T$
is denoted by $u.t$. The set $G = U\times T$ is a groupoid with the following groupoid structure:
$((u,t),(v,t'))$ is composable if and only if $v = u.t, (u, t)(u.t, t')=(u, tt')$ and
$(u, t)^{-l}=(u.t, t^{-1})$. Then $r(u,t) = (u,e)$ and $d(u, t) = (u.t, e)$. If a locally compact group $T$ acts on a locally compact space $U$
 then the
transformation group groupoid
 $U\times T$ with the product topology is a locally compact groupoid.
Now recall that $S_T=C(T,T)$, since $T$ is a group. Let $\varphi\in S_T$, define $f_{\varphi}: G\rightarrow G$ by $f_{\varphi}(u,t)=\Big(u.\big(\varphi(t)\big)^{-1},\varphi(t)\Big)$. Then it is easy to check that $f_{\phi}$ is an element of $S_G$ and it is straightforward to check that $f_{\phi}\ast f_{\psi}=f_{\phi\ast\psi}$ for every $\phi, \psi\in S_T$, and the map $\phi\mapsto f_{\phi}$ is injective, that is the monoid $S_T$  is algebraically isomorphism to a submonoid of $S_G$. Also it is easy to check that the group of units of $S_T$ is embedded in the group of units of $S_G$ by this monomorphism.

  Now for $z\in T$ define $f_{z}(u,t)=(u.z^{-1}, z)$ for all $(u,t)\in G$. We have  $f_{z_1}\ast f_{z_2}=f_{z_1z_2}$. Therefore the set $\{f_{z}: z\in T\}$ with the pointwise topology is a subgroup of $S_G$ which is topologically isomorphic to $T$.
 In a special case if we let $s\in T$ and define $\varphi(t)=ts^{-1}t$ for every $t\in T$, then  $f_{\varphi}$ by
$f_{\varphi}(u,t)=(u.t^{-1}st^{-1},ts^{-1}t)$ is an element of $ T_G$.

\end{example}

\section{A representation of the  elements  of $ S_G$ as  linear operators on $C(G)$}
In the following we will show that every $f\in S_G$ is represented by a linear operator $L_f$ on $C(G)$. Also the map $\Phi:  S_G\rightarrow \mathcal{L}(C(G))$ by $\Phi(f)=L_f$ is a monomorphism, where $\mathcal{L}(C(G))$ is the monoid of all linear operators on $C(G)$ under composition of operators. Moreover the group of units of $S_G$ is embedded in the group of all invertible linear operators on $C(G)$.

\begin{proposition} There exists a map $\phi: S_G\times C(G)\rightarrow C(G)$ with the following properties.
\begin{enumerate}
\item $\phi(f_1\ast f_2, g)=\phi\big(f_1, \phi(f_2,g)\big)$ for all $g\in C(G), f_1, f_2\in S_G$,
\item $\phi(r,g)=g$ for all $g\in C(G)$, where $r(x)=xx^{-1}$.
\item For $f\in S_G$ the map $L_f= \phi(f,.): C(G)\rightarrow C(G)$ is a linear operator and the map $f\mapsto L_f$ from $S_G$ to $\mathcal{L}(C(G))$ is an injective hohomorphism.
\item If $f\in T_G$, the map $L_f$ is an injective linear operator.
\item The group $H(1)$, group of units of $S_G$, is embedded in the group of all invertible linear operators on $C(G)$.
\end{enumerate}
\end{proposition}
From (i) and (ii) we can say that the semigroup $ S_G$ acts on $C(G)$ on the left.
\begin{proof}
For $f\in S_G$ and $g\in C(G)$, define $\phi(f, g):=(g\circ L_f)$, where $(g\circ L_f)(x)=g\big(L_f(x)\big)=g\big(f(x)x\big)$ for all $x\in G$. The map  $g\circ L_f$ is well defined and belongs to $C(G)$. Therefore we have a function from $ S\times C(G)$ to $C(G)$. Now let $f_1, f_2 \in  S_G$ and $g\in C(G)$, then
$$\big(g\circ L_{f_1\ast f_2}\big)=g\circ\big(L_{f_2}\circ L_{f_1}\big)=\Big(g\circ L_{f_2}\big)\circ L_{f_1},$$
and therefore the map $f\mapsto \phi_f$ from $ S_G$ to $\mathcal{L}(C(G))$ is an hohomorphism. This complete the proof of part 1).

   The proof of (2) is straightforward, since $L_r=I$, the identity map on $G$.

   The proof of part (3) is obvious, since the topology of $G$ is Hausdorff.  For part (4), let $f\in T_G,\ g, h\in C(G)$ and $L_f(g)=L_f(h)$. Therefore $g\big(L_f(x)\big)=h\big(L_f(x)\big)$ for every $x\in G$. The density of the set $\{L_f(x): x\in G\}$ in $G$ and the continuity of $g$ and $h$ imply that $g=h$.
Finally (5) is proved by the part (2) and (3).

\end{proof}
There is a similar assertion on the monoid $S'_G$, so we delete it's proof.
\begin{proposition}
There exists a map $\psi:C(G)\times S'_G\rightarrow C(G)$ with the following properties.
\begin{enumerate}

\item $\psi(g,f_1\star f_2)=\psi\big(\psi(g, f_2), f_1\big)$ for all $g\in C(G)$ and all $f_1, f_2\in S'_G$,
\item $\psi(g,d)=g$ for all $g\in C(G)$, where $d(x)=x^{-1}x$.
\item For $f\in S'_G$ the map $R_f=\phi(., f): C(G)\rightarrow C(G)$ is a linear operator.
 and the map $f\mapsto R_f$ from $ S'_G$ to $\mathcal{L}(C(G))$ is an injective hohomorphism.

\item If $f\in T'_G$, the map $R_f$ is an injective linear operator.
\item The group $H'(1)$, group of units of $S'_G$, is embedded in the group of all invertible linear operators on $C(G)$, under the composition of operators.

\end{enumerate}
\end{proposition}
From (i) and (ii) we can say that the semigroup $ S'_G$ acts on $C(G)$ on the right.

\begin{corollary} The semigroup $ S_G$ acts on $C(G)$ on the right. Similarly, the semigroup $ S'_G$ acts on $C(G)$ on the left

\end{corollary}
\begin{proof}
By application of the semigroup isomorphism $f\mapsto f^*$  form $ S_G$ to $ S'_G$ and proposition $4.1, 4.2$, it is obvious.
\end{proof}

It is easy to check that, if $G$ is a principal groupoid, then $ S\cap S'=\{j\}$. The converse is probably true, but I don't have a correct proof.


\bibliographystyle{amsplain}

\end{document}